\documentclass[12pt]{amsart}
\font\bbbld=msbm10 scaled\magstephalf

\newcommand{\bM}{\bar{M}}

\newcommand{\bfR}{\hbox{\bbbld R}}
\newcommand{\bfS}{\hbox{\bbbld S}}

\newcommand{\tF}{\tilde{F}}

\newcommand{\ol}{\overline}
\newcommand{\ul}{\underline}

\newtheorem{theorem}{Theorem}[section]
\newtheorem{lemma}[theorem]{Lemma}
\newtheorem{proposition}[theorem]{Proposition}
\newtheorem{corollary}[theorem]{Corollary}

 \theoremstyle{definition}

\theoremstyle{remark}
\newtheorem{remark}[theorem]{Remark}

\numberwithin{equation}{section}

\begin{document}
\setlength{\baselineskip}{1.2\baselineskip}

\title[Fully Nonlinear Elliptic Equations]
{Second Order Estimates and Regularity for \\
Fully Nonlinear Elliptic Equations \\ on Riemannian Manifolds}
\author{Bo Guan}
\address{Department of Mathematics, Ohio State University,
         Columbus, OH 43210}
\email{guan@math.osu.edu}

\begin{abstract}
We derive {\em a priori} second order estimates for solutions of a
class of fully nonlinear elliptic equations on Riemannian manifolds
under some very general structure conditions.
We treat both equations on closed manifolds, and the Dirichlet problem
on manifolds with boundary without any geometric restrictions to the 
boundary except being smooth and compact. As applications of these estimates
we obtain results on regularity and existence. 

{\em Mathematical Subject Classification (2010):}
  35B45, 35J15, 58J05.

{\em Keywords:} Fully nonlinear elliptic equations on Riemnnian manifolds;
{\em a priori} estimates; Dirichlet problem; subsolutions;
strict concavity property.

\end{abstract}

\maketitle

\bigskip

\section{Introduction}
\label{hess-I}
\setcounter{equation}{0}

\medskip

This is one of several papers in which we seek methods to derive {\em a priori}
estimates for fully nonlinear elliptic equations on real or complex manifolds.
Our techniques work for various classes of equations under conditions which
are near optimal in many situations.
In this paper we shall focus on the second order estimates for the
Hessian type equations on Riemannian manifolds.

Let $(M^n, g)$ be a compact Riemannian manifold of
dimension $n \geq 2$ with smooth boundary $\partial M$, and
 $\bM := M \cup \partial M$.
Let $f$ be a smooth symmetric function of $n$ variables and
$\chi$ a smooth $(0,2)$ tensor on  $\bM$.
We consider fully nonlinear equations of the form
\begin{equation}
\label{3I-10}
f (\lambda [\nabla^2 u + \chi]) = \psi \;\; \mbox{in $M$}
\end{equation}
where $\nabla^2 u$ denotes the Hessian of $u \in C^2 (M)$
and $\lambda [\nabla^2 u + \chi] = (\lambda_1, \cdots, \lambda_n)$ are
the eigenvalues of $\nabla^2 u + \chi$ with respect to the metric $g$.

Fully nonlinear equations of form \eqref{3I-10}
in $\bfR^n$ was first considered by Caffarelli, Nirenberg and Spruck in their
seminal paper \cite{CNS3}.
Following \cite{CNS3} we assume $f$ is defined
in a symmetric open and convex cone $\Gamma \subset \bfR^n$ with vertex
at the origin and boundary $\partial \Gamma \neq \emptyset$,
\begin{equation}
\label{3I-15}
\Gamma^+ \equiv
\{\lambda \in \bfR^n: \mbox{each component $\lambda_i > 0$}\} \subseteq \Gamma,
\end{equation}
and to satisfy the standard structure conditions:
\begin{equation}
\label{3I-20}
f_i = f_{\lambda_i} \equiv \frac{\partial f}{\partial \lambda_i} > 0 \;\;\;
\mbox{in $\Gamma$}, \;\;\;\; 1 \leq i \leq n,
\end{equation}
\begin{equation}
\label{3I-30}
\mbox{$f$ is a concave function},
\end{equation}
\begin{equation}
\label{3I-40}
\delta_{\psi, f} \equiv \inf \psi - \sup_{\partial \Gamma} f > 0; \;\;
 \mbox{where} \; \sup_{\partial \Gamma} f \equiv \sup_{\lambda_0 \in \partial \Gamma}
                   \limsup_{\lambda \rightarrow \lambda_0} f (\lambda).
\end{equation}

According to \cite{CNS3}
condition (\ref{3I-20}) ensures that equation (\ref{3I-10}) is elliptic
for solutions $u \in C^{2} (M)$ with $\lambda [\nabla^2 u + \chi] \in \Gamma$;
we shall call such functions {\em admissible},
while condition \eqref{3I-30} implies the function $F$ defined by
$F (A) = f (\lambda [A])$ to be concave for $A \in \mathcal{S}^{n\times n}$ with
$\lambda [A] \in \Gamma$, where $\mathcal{S}^{n\times n}$ is the set of
$n$ by $n$ symmetric matrices. By condition~\eqref{3I-40},
 equation~\eqref{3I-10} {\em becomes} uniformly elliptic once {\em a priori}
$C^2$ bounds are established for admissible solutions so that
one can apply the classical Evans-Krylov theorem to obtain $C^{2,\alpha}$
estimates. So these conditions are basically indispensable to the study of
equation~\eqref{3I-10}.

The most typical equations of form \eqref{3I-10} are given by
 $f = \sigma_k^{\frac{1}{k}}$ and
$f = (\sigma_k/\sigma_l)^{\frac{1}{k-l}}$, $1 \leq l <  k \leq n$ 
defined on the cone
\[ \Gamma_k = \{\lambda \in \bfR^n: \sigma_j (\lambda) > 0 \;
\mbox{for $1 \leq j \leq k$}\}, \]
where $\sigma_k$ is the $k$-th elementary symmetric
function 
\[ \sigma_k (\lambda) = \sum_{i_1 < \cdots < i_k}
\lambda_{i_1} \cdots \lambda_{i_k}, \;\; 1 \leq k \leq n.\]
These functions satisfy \eqref{3I-20}-\eqref{3I-30} and have other properties
which have been widely used in study of the corresponding equations;
see e.g. \cite{CNS3},
\cite{LiYY90}, \cite{Trudinger90}, \cite{LT94}, \cite{Wang94}, \cite{CW01}.

The Dirichlet problem for equation~\eqref{3I-10} in $\bfR^n$ was
extensively studied by
Caffarelli, Nirenberg and Spruck~\cite{CNS3}, Ivochkina~\cite{Ivochkina85},
Krylov~\cite{Krylov83}, Wang~\cite{Wang94}, Trudinger~\cite{Trudinger95},
Trudinger and Wang~\cite{TW99}, Chou and Wang~\cite{CW01},
and the author~\cite{Guan94}, ~\cite{Guan}, among many others.
In this paper we deal with equation~\eqref{3I-10} on general Riemannian manifolds.

Equation~\eqref{3I-10} was first studied by
Y.-Y. Li~\cite{LiYY90} on closed Riemannian manifolds,
followed by the work of Urbas~\cite{Urbas02}.

A central issue in solving equation~\eqref{3I-10} is to derive $C^2$ estimates
for admissible solutions, in view of the Evans-Krylov theorem.
We shall be mainly concerned with estimates for second derivatives.
Such estimates was first derived by Y.-Y. Li~\cite{LiYY90} for
equation~\eqref{3I-10}
with $\chi = g$ on closed manifolds of nonnegative sectional curvature.
Urbas~\cite{Urbas02} was able to remove the nonnegative curvature assumption.
In deriving the estimates, the presence of curvature creates terms which are
difficult to control.
As a result, in addition to (\ref{3I-20})-(\ref{3I-40})
both papers needed extra assumptions which excluded the case
$f = (\sigma_k/\sigma_l)^{1/(k-l)}$; see Section~\ref{3I-R} for more discussions
about the results of ~\cite{LiYY90} and \cite{Urbas02}.

In order to state our main results,
which cover the case $f = (\sigma_k/\sigma_l)^{1/(k-l)}$,
we first introduce some notation.

For $\sigma > \sup_{\partial \Gamma} f$, define
$\Gamma^{\sigma} = \{\lambda \in \Gamma: f (\lambda) > \sigma\}$, and
we shall only consider the case $\Gamma^{\sigma} \neq \emptyset$.
Let $\mathcal{C}_{\sigma}$ denote the tangent cone at infinity
to the level surface $\partial \Gamma^{\sigma}$ which is smooth and convex
by conditions~(\ref{3I-20}) and (\ref{3I-30}).
Let $\mathcal{C}_{\sigma}^+$ be the open component of
$\Gamma \setminus (\mathcal{C}_{\sigma} \cap \Gamma)$ containing $\Gamma^{\sigma}$.

Our first main result is the following global second order estimates.

\begin{theorem}
\label{3I-th4}
Let $\psi \in C^2 (M \times \bfR) \cap C^1 (\bar{M}  \times \bfR)$ and
$u \in C^4 (M) \cap C^2 (\bar{M})$ be an admissible solution of \eqref{3I-10}.
Suppose $a \leq u \leq b$ on $\bM$ and let
\[ \ul{\psi} (x) = \min_{a \leq z \leq b} \psi (x, z), \;\;
   \hat{\psi} (x) = \max_{a \leq z \leq b} \psi (x, z), \;\; x \in \bM. \]
In addition to (\ref{3I-20})-(\ref{3I-30}), assume
\begin{equation}
\label{3I-40'}
  \delta_{\ul{\psi}, f} = \inf_{\bM} \ul{\psi} - \sup_{\partial \Gamma} f > 0.
\end{equation}
and that there exists a function $\ul{u} \in C^2 (\bar{M})$ satisfying
\begin{equation}
\label{3I-200}
\lambda [\nabla^2 \ul{u} + \chi] (x) \in \mathcal{C}_{\hat{\psi} (x)}^+,
\;\; \forall \; x \in \bM.
\end{equation}
Then
\begin{equation}
\label{hess-a10}
\max_{\bar{M}} |\nabla^2 u| \leq
 C_1 \big(1 + \max_{\partial M}|\nabla^2 u|\big).
\end{equation}
In particular, if
$M$ is closed ($\partial M = \emptyset)$ then
\begin{equation}
\label{hess-a10c}
|\nabla^2 u| \leq C_2 e^{C_3 (u - \inf_{M} u)} \;\; \mbox{on $M$}
\end{equation}
where $C_1$, $C_2$ depend on $|u|_{C^1 (M)}$ but not on $1/\delta_{\ul{\psi}, f}$
and $C_3$ is a uniform constant (independent of $u$).
\end{theorem}

As we shall see in Section~\ref{3I-R}, condition~\eqref{3I-200} is implied by
the assumptions in \cite{LiYY90}.
By approximation we obtain the following regularity result from
Theorem~\ref{3I-th4}.

\begin{theorem}
\label{3I-th5}
Let $(M^n, g)$ be a closed Riemannian manifold and
$\psi \in C^{1,1} (M \times \bfR)$.
Under conditions (\ref{3I-20})-(\ref{3I-30}), \eqref{3I-40}
and \eqref{3I-200}, any admissible weak solution (in the
viscosity sense) $u \in C^{0,1} (M)$ of \eqref{3I-10}
belongs to $C^{1,1} (M)$ and \eqref{hess-a10c} holds.
\end{theorem}

By the Evans-Krylov theorem, $u \in C^{2,\alpha} (M)$, $0 < \alpha < 1$;
higher regularities follow from the classical Schauder elliptic theory.
In particular, $u \in C^{\infty} (M)$ if $\psi \in C^{\infty} (M)$.

\begin{remark}
\label{3I-remark20}
Condition~\eqref{3I-200} is always satisfied if there is a strictly
convex function on $M$ ($\partial M \neq \emptyset$), or if
$\chi \in \mathcal{C}_{\sigma}^+$ (for instance, if
$\chi = a g$, $a > 0$ and the vertex of $\mathcal{C}_{\sigma}$ is
the origin) for all $\sigma$.
For $f = \sigma_k^{1/k}$ ($k \geq 2$), $\Gamma_n^+ \subset \mathcal{C}_{\sigma}^+$
for any $\sigma > 0$. See also Lemma~\ref{3I-lemma-R10}.
\end{remark}

\begin{corollary}
\label{3I-cor1}
Let $(M, g)$ be a closed Riemannian manifold and $\psi \in C^{1,1} (M)$.
In addition to (\ref{3I-20})-(\ref{3I-40}), suppose
$\chi \in \mathcal{C}_{\sigma}^+$ for all
$\sup_{\partial \Gamma} f < \sigma \leq \sup_M \psi$.
Then any admissible weak solution $u \in C^{0,1} (M)$ of \eqref{3I-10}
belongs to $C^{2,\alpha} (M)$, $0 < \alpha < 1$, and \eqref{hess-a10c} holds.
\end{corollary}

We now turn to the second order boundary estimates. We wish to derive
such estimates without imposing any geometric conditions
on $\partial M$ except being smooth and compact.
For simplicity we only consider the case $\psi = \psi (x)$.

\begin{theorem}
\label{3I-th4b}
Let $\psi \in C^1 (\bar{M})$, $\varphi  \in C^4 (\partial M)$
and $u \in C^3 (M) \cap C^1 (\bar{M})$ be an admissible solution of \eqref{3I-10}
with $u = \varphi$ on $\partial M$.
Assume $f$ satisfies (\ref{3I-20})-(\ref{3I-40}) and
\begin{equation}
\label{3I-55}
\sum f_i \, \lambda_i \geq 0 \;\; \mbox{in $\Gamma$}.
\end{equation}
Suppose that there exists an admissible subsolution
$\ul{u} \in C^0 (\bar{M})$  in the viscosity sense:
\begin{equation}
\label{3I-11s}
\left\{ \begin{aligned}
f (\lambda [\ul{u}_{ij} + a_{ij}]) \,& \geq \psi \;\; \mbox{in $\bM$}, \\
   \ul{u} = \varphi \;\; \,& \mbox{on $\partial M$} \end{aligned} \right.
\end{equation}
 and that $\ul{u}$ is $C^2$ and satisfies
\begin{equation}
\label{3I-11s''}
  \lambda [\nabla^2 \ul{u} + \chi] (x) \in \mathcal{C}^+_{\psi (x)}
\end{equation}
in a neighborhood of $\partial M$.
Then there exists $C_4 > 0$ depending on $|u|_{C^1 (\bM)}$
and $1/\delta_{\ul{\psi}, f}$ such that
\begin{equation}
\label{hess-a10b}
\max_{\partial M}|\nabla^2 u| \leq C_4.
\end{equation}
\end{theorem}

\begin{remark}
An admissible subsolution $\ul{u} \in C^2 (\bar{M})$
will automatically satisfy \eqref{3I-200}
provided that
\begin{equation}
\label{3I-30sc}
\partial \Gamma^{\sigma} \cap \mathcal{C}_{\sigma} = \emptyset,
\;\; \forall \; \sigma \in \big[\inf_M \psi, \sup_M \psi \big].
\end{equation}
Condition~\eqref{3I-30sc} excludes the linear function $f = \sigma_1$ which
corresponds to the Poisson equation, but is clearly satisfied by a wide class
of concave functions including $f = \sigma_k^{1/k}$,
$k \geq 2$ and $f = (\sigma_k/\sigma_l)^{1/(k-l)}$
for all $1 \leq l < k \leq n$.
Note that condition~\eqref{3I-30sc} holds if $\partial \Gamma^{\sigma}$ is
strictly convex.
\end{remark}

Applying Theorems~\ref{3I-th4} and \ref{3I-th4b} we can prove the following
existence result by the standard continuity method.

\begin{theorem}
\label{3I-thm2}
Let $\psi \in C^{\infty} (\bM)$, $\varphi \in C^{\infty} (\partial M)$.
Suppose $f$ satisfies (\ref{3I-20})-(\ref{3I-40}), \eqref{3I-55} and that
there exists an admissible subsolution
$\ul{u} \in C^2 (\bM)$ satisfying \eqref{3I-11s}
and \eqref{3I-11s''} for all $x \in \bM$.
Then there exists an admissible solution $u \in C^{\infty} (\bM)$ of the Dirichlet
problem for equation~\eqref{3I-10} with boundary condition
$u = \varphi$ on $\partial M$, provided that
{\bf (i)} $\Gamma = \Gamma_n^+$, or
{\bf (ii)} the sectional curvature of $(M, g)$ is nonnegative, or
{\bf (iii)} $f$ satisfies
\begin{equation}
\label{3I-R40}
  f_j \geq \delta_0 \sum f_i (\lambda) \; \mbox{if $\lambda_j < 0$,
on $\partial \Gamma^{\sigma} \; \forall \; \sigma > \sup_{\partial \Gamma} f$}.
\end{equation}
\end{theorem}

When $M$ is a smooth bounded domain in $\bfR^n$, Theorem~\ref{3I-thm2} {\bf (ii)}
extends the previous results of Caffarelli, Nirenberg and Spruck~\cite{CNS3}, 
Trudinger~\cite{Trudinger95} and the author~\cite{Guan94}; see ~\cite{Guan} for more 
detailed discussions.
The assumptions {\bf (i)}-{\bf (iii)} are only needed to derive gradient
estimates; see Proposition~\ref{3I-propositio-R20}. It would be desirable to
remove these assumptions.

\begin{corollary}
\label{3I-cor-R2}
Let $f = \sigma_k^{{1}/{k}}$, $k \geq 2$ or
$f = (\sigma_k/\sigma_l)^{\frac{1}{k-l}}$, $0 \leq l <  k \leq n$.
Given $\psi \in C^{\infty} (\bM)$, $\psi > 0$ and
$\varphi \in C^{\infty} (\partial M)$,
suppose that there exists an admissible subsolution
$\ul{u} \in C^2 (\bM)$ satisfying \eqref{3I-11s}.
Then there exists an admissible solution $u \in C^{\infty} (\bM)$ of
equation~\eqref{3I-10}
with $u = \varphi$ on $\partial M$.
\end{corollary}

In Theorem~\ref{3I-thm2} there are no geometric restrictions
to $\partial M$ being made. This gives Theorem~\ref{3I-thm2} the advantage of
flexibility in applications. In general, the Dirichlet problem
is not always solvable in arbitrary domains without the subsolution assumption,
as in the case of Monge-Amp\`ere equations.
In the classical theory of elliptic equations, a standard technique
is to use the distance function to the boundary to construct local barriers for
boundary estimates. So one usually need require the boundary to
possess certain geometric properties; see e.g. \cite{Serrin69} for
the prescribed mean curvature equation and \cite{CNS1}, \cite{CKNS} for
Monge-Amp\`ere equations; see also \cite{GT} and \cite{CNS3}.
Technically, we use $u - \ul{u}$ to replace the boundary distance function
in deriving the second order boundary estimates.
This idea was first used by Haffman, Rosenberg and Spruck~\cite{HRS}
and further developed in \cite{GS93}, \cite{GL96},
\cite{Guan98a}, \cite{Guan98b} to treat the real and complex Monge-Amp\`ere
equations in general domains as well as in \cite{Guan94}, \cite{Guan99a}
for more general fully nonlinear equations.
 Their results and techniques have found useful
applications in some important problems; see e.g. the work of
P.-F. Guan~\cite{GuanPF02}, \cite{GuanPF10} and  papers
of Chen~\cite{Chen00}, Blocki~\cite{Blocki}, and Phong and Sturm~\cite{PS07}
on the Donaldson conjectures~\cite{Donaldson99} in K\"ahler geometry.
In \cite{GS93}, \cite{GS02},
\cite{GS04} we used the techniques to study Plateau type problems for
locally convex hypersurfaces of constant curvature in $\bfR^{n+1}$.

We shall also make use of $u - \ul{u}$ in the proof of the global
estimate~\eqref{hess-a10}.
This is one of the key ideas in this paper; see the proof in 
Section~\ref{hess-g}.
Note that in Theorem~\ref{3I-th4} the function $\ul{u}$ is not necessarily 
a subsolution.
On a closed manifold, an admissible subsolution for $\psi = \psi (x)$ must 
be a solution if there is a solution at all, and any two
admissible solutions differ at most by a constant. This is a consequence of
the concavity condition~\eqref{3I-30} and the maximum principle.

Similar equations where $\chi$ depends on $u$ or $\nabla u$ (or both) also occur
naturally and have received extensive study in classical differential geometry;
see e.g. \cite{GG02}, \cite{GM03}, and in conformal geometry in which
there is a huge literature; see for instance \cite{CGY02a}, \cite{CGY02b},
\cite{ChenS05}, 
~\cite{GeW}, \cite{GW03a}, \cite{GW03b}, 
\cite{GV03a}, 
\cite{GV1}, \cite{Han04}, \cite{LL03}, \cite{LL05}, \cite{LS05}, ~\cite{STW},
\cite{V00}, \cite{V02} and references therein.
In the current paper we confine our discussion to the case $\chi = \chi (x)$,
$x \in \bM$.

In Section~\ref{3I-C} we discuss some consequences of the concavity
condition.
Our proof of the estimates heavily depends on results in Section~\ref{3I-C}.
The global and boundary estimates are derived in
Sections~\ref{hess-g} and \ref{hess-b}, respectively.
In Section~\ref{3I-R} we briefly discuss the results of Li \cite{LiYY90} and
Urbas~\cite{Urbas02}, followed by gradient estimates. We end the paper with a new
example which was first brought to our attention by Xinan Ma to whom we wish to
express our gratitude.

The author also wishes to thank Jiaping Wang for helpful discussions on the proof
of Theorem~\ref{3I-th3} and related topics.

\bigskip

\section{The concavity condition}
\label{3I-C}
\setcounter{equation}{0}

\medskip

Let $\sigma > \sup_{\partial \Gamma} f$ and assume
$\Gamma^{\sigma} := \{f > \sigma\} \neq \emptyset$. Then
$\partial \Gamma^{\sigma}$ is a smooth convex noncompact complete
hypersurface contained in $\Gamma$.
Clearly $\Gamma^{\sigma} \neq \mathcal{C}_{\sigma}^+$ unless
$\partial \Gamma^{\sigma}$ is a plane.

Let $\mu, \lambda \in \partial \Gamma^{\sigma}$. By the convexity of
$\partial \Gamma^{\sigma}$, the open segment
\[ (\mu, \lambda) \equiv \{t \mu + (1-t)\lambda: 0 < t < 1\} \]
is either completely contained in or does not intersect with
$\partial \Gamma^{\sigma}$. Therefore,
\[ f (t \mu + (1-t) \lambda) - \sigma > 0, \;\; \forall \, 0 < t < 1 \]
by condition~\eqref{3I-20}, unless
$(\mu, \lambda) \subset \partial \Gamma^{\sigma}$.

For $R > |\mu|$, let
\[ \Theta_R (\mu) \equiv
  \inf_{\lambda \in \partial B_R (0) \cap \partial \Gamma^{\sigma}}
        \max_{0 \leq t \leq 1} f (t \mu + (1-t)\lambda) - \sigma \geq 0. \]
Note that $\Theta_R (\mu) = 0$ if and only if
$(\mu, \lambda) \subset \partial \Gamma^{\sigma}$
for some $\lambda \in \partial B_R (0) \cap \partial \Gamma^{\sigma}$,
since the set $\partial B_R (0) \cap \partial \Gamma^{\sigma}$ is compact.

\begin{lemma}
\label{3I-C-lemma10}
For $\mu \in \partial \Gamma^{\sigma}$, $\Theta_R (\mu)$ is nondecreasing in
$R$. Moreover, if $\Theta_{R_0} (\mu) > 0$ for some $R_0 \geq |\mu|$ then
$\Theta_{R'} > \Theta_{R}$ for all $R' > R \geq R_0$.
\end{lemma}

\begin{proof}
Write
$\Theta_R = \Theta_R (\mu)$ when there is no possible confusion.
Suppose $\Theta_{R_0} (\mu) > 0$ for some $R_0 \geq |\mu|$.
Let $R' > R \geq R_0$ and assume 
$\lambda_{R'} \in \partial B_{R'} (0) \cap \partial \Gamma^{\sigma}$
such that
\[     \Theta_{R'} =
        \max_{0 \leq t \leq 1} f (t \mu + (1-t) \lambda_{R'}) - \sigma. \]
Let $P$ be the (two dimensional) plane through $\mu, \lambda_{R'}$ and the
origin of $\bfR^n$. There is a point $\lambda_R \in \partial B_R (0)$ which
lies between $\mu$ and $\lambda_{R'}$ on the curve $P \cap \partial \Gamma^{\sigma}$.
Note that $\mu$, $\lambda_R$ and $\lambda_R'$ are not on a straight line,
for $(\mu, \lambda_{R})$ can not be part of $(\mu, \lambda_{R'})$
since $\Theta_{R_0} > 0$ and $\partial \Gamma^{\sigma}$ is convex.
We see that
\[ \max_{0 \leq t \leq 1} f (t \mu + (1-t) \lambda_{R}) - \sigma
   < \Theta_{R'} \]
by condition~\eqref{3I-20}. This proves $\Theta_{R} < \Theta_{R'}$.
\end{proof}

\begin{corollary}
\label{3I-cor20}
Let $\mu \in \partial \Gamma^{\sigma}$.
The following are equivalent:

{\bf (a)} $\mu \in \mathcal{C}_{\sigma}$;

{\bf (b)} $\Theta_R (\mu) = 0$ for all $R > |\mu|$;

{\bf (c)} $\partial \Gamma^{\sigma} \cap \mathcal{C}_{\sigma}$
          contains a ray through $\mu$;

{\bf (d)} $T_{\mu} \partial \Gamma^{\sigma} \cap \mathcal{C}_{\sigma}$
          contains a ray through $\mu$, where $T_{\mu} \partial \Gamma^{\sigma}$
          is the tangent (supporting) plane of $\partial \Gamma^{\sigma}$ at $\mu$.
\end{corollary}

\begin{lemma}
\label{gblq-lemma-C20}
Let $\mu \in \ol{\Gamma^{\sigma}}$, $\mu \notin \mathcal{C}_{\sigma}$.
There exist positive constants
$\omega_{\mu}$, $N_{\mu}$ such that for any
$\lambda \in \partial \Gamma^{\sigma}$, when $|\lambda| \geq N_{\mu}$,
\begin{equation}
\label{gblq-C210}
\sum f_i (\lambda) (\mu_{i} - \lambda_i) \geq \omega_{\mu}.
\end{equation}
\end{lemma}

\begin{proof}
By the concavity of $f$,
\[ \sum f_i (\lambda) (\mu_{i} - \lambda_i)
    \geq f (\mu) - f (\lambda). \]
We see \eqref{gblq-C210} holds if $f (\mu) > \sigma$.
So we assume $\mu \in \partial \Gamma^{\sigma}$.
By Corollary~\ref{3I-cor20}, $\Theta_R (\mu) > 0$ for all $R$ sufficiently large,
and therefore, again by the concavity of $f$,
\[ \sum f_i (\lambda) (\mu_{i} - \lambda_i)
    \geq \max_{0 \leq t \leq 1} f (t \mu + (1-t) \lambda)
     - \sigma \geq \Theta_R (\mu) > 0 \]
for any $\lambda \in \partial B_R (0) \cap \partial \Gamma^{\sigma}$.
Since $\Theta_R (\mu)$ is increasing in $R$, Lemma~\ref{gblq-lemma-C20}
holds.
\end{proof}

Our main results of this paper is based on the following observation.

\begin{theorem}
\label{3I-th3}
Let $\mu \in \mathcal{C}_{\sigma}^+$.
For any $0 < \varepsilon < \mbox{dist} (\mu, \mathcal{C}_{\sigma})$
there exist positive constants
$\theta_{\mu}$, $R_{\mu}$ such that for any
$\lambda \in \partial \Gamma^{\sigma}$, when $|\lambda| \geq R_{\mu}$,
\begin{equation}
\label{3I-100}
\sum f_i (\lambda) (\mu_{i} - \lambda_i)
    \geq \theta_{\mu} + \varepsilon \sum f_i (\lambda).
\end{equation}
\end{theorem}

\begin{proof} 
Since $\mu \in \mathcal{C}_{\sigma}^+$
and $\varepsilon < \mbox{dist} (\mu, \mathcal{C}_{\sigma})$, we see that
$\mu^{\varepsilon} \equiv \mu - \varepsilon {\bf 1} \in \mathcal{C}_{\sigma}^+$
where ${\bf 1} = (1, \ldots, 1)$.
Let $\mathcal{C} ({\mu^{\varepsilon}})$ be the tangent cone to
 $\Gamma^{\sigma}$ with vertex $\mu^{\varepsilon}$.
Then $\partial \Gamma^{\sigma} \cap \mathcal{C} (\mu^{\varepsilon})$
is compact and therefore contained in a ball $B_{R_0} (0)$ for some $R_0 > 0$.
Let $\partial \Gamma_{\sigma, \mu^{\varepsilon}}$ denote the compact subset of
$\partial \Gamma^{\sigma}$ bounded by
$\partial \Gamma^{\sigma} \cap \mathcal{C} (\mu^{\varepsilon})$.

Let $R > R_0$ and $\lambda \in \partial B_{R} (0) \cap \partial \Gamma^{\sigma}$.
The segment $[\mu^{\varepsilon}, \lambda]$ goes through
$\partial \Gamma_{\sigma, \mu^{\varepsilon}}$ at a point $\lambda^{\varepsilon}$.
Since $f (\lambda) = f (\lambda^{\varepsilon}) = \sigma$,
by the concavity of $f$ we obtain
\[ \begin{aligned}
\sum f_i (\lambda) ((\mu_{i} - \varepsilon) - \lambda_i)
   \geq \,& \sum f_i (\lambda) (\lambda^\varepsilon_i - \lambda_i)
  \geq \omega_{\lambda^\varepsilon}
  \geq \inf_{\eta \in \partial \Gamma_{\sigma, \mu^{\varepsilon}}} \omega_{\eta}
   \equiv \theta_{\mu} > 0
\end{aligned} \]
when $R \geq R_{\mu} \equiv
\sup_{\eta \in \partial \Gamma_{\sigma, \mu^{\varepsilon}}} N_{\eta}$.
\end{proof}

Theorem~\ref{3I-th3} can not be used directly in the proofs
of \eqref{hess-a10} and \eqref{hess-a10b} in the next two sections.
So we modify it as follows.

Let $\mathcal{A}$ be the set of $n$ by $n$ symmetric
matrices $A = \{A_{ij}\}$ with eigenvalues $\lambda [A] \in \Gamma$.
Define the function $F$ on $\mathcal{A}$ by
\[ F (A) \equiv f (\lambda [A]). \]
Throughout this paper we shall use the notation
\[ F^{ij} (A) = \frac{\partial F}{\partial A_{ij}} (A), \;\;
  F^{ij, kl} (A) = \frac{\partial^2 F}{\partial A_{ij} \partial A_{kl}} (A). \]
The matrix $\{F^{ij}\}$ has eigenvalues $f_1, \ldots, f_n$ and
is positive definite by assumption (\ref{3I-20}), while (\ref{3I-30})
implies that $F$ is a concave function of $A_{ij}$ \cite{CNS3}.
Moreover, when $A$ is diagonal so is $\{F^{ij} (A)\}$, and the
  following identities hold
\[   F^{ij} (A) A_{ij} = \sum f_i \lambda_i, \]
\[  F^{ij} (A) A_{ik} A_{kj} = \sum f_i \lambda_i^2. \]

\begin{theorem}
\label{3I-th3'}
Let $A \in \mathcal{A}$, $\lambda (A) \in \mathcal{C}_{\sigma}^+$.
Then for any $0 < \varepsilon < \mbox{dist} (\lambda (A), \mathcal{C}_{\sigma})$
there exist positive constants
$\theta_{A}$, $R_{A}$ such that for any
$B \in \mathcal{A}$ with $\lambda (B) \in \partial \Gamma^{\sigma}$,
when $|\lambda (B)| \geq R_{A}$,
\begin{equation}
\label{3I-100'}
F^{ij} (B) (A_{ij} - B_{ij})
    \geq \theta_{A} + \varepsilon \sum F^{ii} (B).
\end{equation}
\end{theorem}

\begin{proof}
Suppose first that $\lambda (A) \in \Gamma^{\sigma}$. Then,
since $\lambda (A) \notin \mathcal{C}_{\sigma}$,
\[ (A, B) \equiv \{t A + (1-t) B: 0 < t < 1\} \]
is completely contained in $\Gamma^{\sigma}$ for any $B \in \mathcal{A}$ with
$\lambda (B) \in \partial B_R (0) \cap \partial \Gamma^{\sigma}$
when $R$ is sufficiently large. Therefore,
\[ \Theta_R (A) \equiv \inf_{\lambda (B) \in \partial B_R (0) \cap \partial \Gamma^{\sigma}}
        \max_{0 \leq t \leq 1} F (t A + (1-t) B) - \sigma > 0 \]
and $\Theta_R (A)$ is increasing in $R$. By the concavity of $F$ we have
\[ F^{ij} (B) (A_{ij} - B_{ij})
   \geq \max_{0 \leq t \leq 1} F (t A + (1-t) B) - \sigma \geq \Theta_R (A) \]

In the general case, let
$A^{\varepsilon} = A - \varepsilon I \in \mathcal{A}$
so $\lambda (A^{\varepsilon}) = \lambda (A) - \varepsilon {\bf 1}$.
When $R$ is sufficiently large, for any $B \in \mathcal{A}$ with
$\lambda (B) \in \partial B_R (0) \cap \partial \Gamma^{\sigma}$
we can find $C \in (A, B)$ such that $\lambda (C)$ is contained in
the compact set $\partial \Gamma_{\sigma, \lambda (A^{\varepsilon})}$.
As before,
\[ F^{ij} (B) (A_{ij} - \varepsilon \delta_{ij} - B_{ij})
\geq F^{ij} (B) (C_{ij} - B_{ij}) \geq \Theta_R (C). \]
This completes the proof of Theorem~\ref{3I-th3'}
in view of the compactness of
$\partial \Gamma_{\sigma, \lambda (A^{\varepsilon})}$.
\end{proof}

The following inequality is taken from \cite{GS10} with minor modifications.
We shall need it in the boundary estimates in Section~\ref{hess-b}.

\begin{proposition}
\label{prop5.2}
Let $A = \{A_{ij}\} \in \mathcal{A}$ and set $F^{ij} =  F^{ij} (A)$.
There is $c_0 > 0$ and an index $r$ such that
 \begin{equation}
 \label{eq5.170}
  \sum_{l<n} F^{ij} A_{il} A_{lj} \geq c_0 \sum_{i \neq r}  f_i \lambda_i^2.
\end{equation}
\end{proposition}

\begin{proof}
 Let $B = \{b_{ij}\}$ be an orthogonal matrix that simultaneously
diagonalizes
 $\{F^{ij}\}$ and $\{A_{ij}\}$:
 \[ F^{ij} b_{li} b_{kj} = f_k \delta_{kl}, \;\;
    A_{ij} b_{li} b_{kj} = \lambda_k \delta_{kl}. \]
Then 
\begin{equation}
\label{eq5.180}
\begin{aligned}
\sum_{l<n} F^{ij} A_{il} A_{lj}
 = & \, \sum_{l<n} f_i \lambda_i^2 b_{li}^2.
\end{aligned}
\end{equation}

Suppose for some $i$, say $i=1$ and $0 < \theta < 1$ to be determined
that
\[ \sum_{l<n} b_{l1}^2 < \theta^2. \]
Then
\[ b_{n1}^2 = 1 -  \sum_{l<n} b_{l1}^2 > 1 - \theta^2 > 0. \]
Expanding $\det B$ by cofactors along the first column gives
\[1 = \det{B} = b_{11}C^{11} 
          +\ldots + b_{(n-1)1}C^{1 n-1}+ b_{n1}\det{D}
     \leq c_1 \theta + |b_{n1} \det D|, \]
where $C^{1j}$ are the cofactors and $D$ is the $n-1$ by $n-1$ matrix
\begin{equation}
D = \begin{bmatrix}
 b_{12}  & \dots   & b_{(n-1)2}  \\
 \vdots  & \ddots  & \vdots \\
 b_{1n}  & \ldots  & b_{(n-1)n}
   \end{bmatrix}.
\end{equation}
Therefore,
\[ |\det{D}|\geq \frac{1-c_1 \theta}{|b_{n1}|} \geq 1-c_1 \theta. \]
Now expanding  $\det {D}$
by cofactors along row $i \geq 2$ gives
\[ |\det D| \leq c_2 \Big(\sum_{l<n}b_{li}^2\Big)^{\frac{1}{2}} \]
by Schwarz inequality. Hence
\begin{equation}
\label{eq5.190}
 \sum_{l<n}b_{li}^2 \geq
    \Big(\frac{1-c_1 \theta}{c_2}\Big)^2.
\end{equation}
 Choosing $\theta < \frac1{2c_1}$, \eqref{eq5.190} and \eqref{eq5.180}
imply
\[ \sum_{l<n} F^{ij} A_{il} A_{lj} \geq c_0 \sum_{i\neq 1} f_i \lambda_i^2. \]
This proves \eqref{eq5.170}.
\end{proof}

\begin{lemma}
\label{lem1.10}
 Suppose f satisfies (\ref{3I-20}), (\ref{3I-30}) and (\ref{3I-55}). Then
\begin{equation}
 \sum_{i\neq r} f_i \lambda_i^2  \geq \frac{1}{n}  \sum f_i \lambda_i^2
  \;\; \mbox{if $\lambda_r<0$}.
\end{equation}
\end{lemma}

\begin{proof}
Suppose $\lambda_1 \geq \cdots \geq \lambda_n$ and $\lambda_r < 0 $.
By the concavity condition (\ref{3I-30}) we have
$f_n \geq f_i > 0$ for all $i$ and in particular
$f_n \lambda_n^2 \geq  f_r \lambda_r^2$.
By (\ref{3I-55}),
 \[ \sum_{i\neq n}f_i \lambda_i \geq - f_n \lambda_n = f_n |\lambda_n|. \]
By Schwarz inequality,
\[ f_n^2 \lambda_n^2
   \leq \sum_{i\neq n} f_i \sum_{i \neq n} f_i \lambda_i^2
   \leq (n-1) f_n \sum_{i \neq n}  f_i \lambda_i^2. \]
 Therefore,
\[ \sum_{i \neq r} f_i \lambda_i^2 \geq \sum_{i \neq n}  f_i \lambda_i^2
     \geq \frac{1}{n} \sum_{i \neq n}  f_i \lambda_i^2 + \frac{1}{n} f_n \lambda_n^2
      = \frac{1}{n} \sum f_i \lambda_i^2 \]
completing the proof.
\end{proof}

\begin{corollary}
\label{cor1.10}
 Suppose f satisfies (\ref{3I-20})-(\ref{3I-30}). Then for any index $r$
\begin{equation}
\label{C210}
\sum f_i |\lambda_i| \leq \epsilon \sum_{i\neq r} f_i \lambda_i^2
  + C \Big(1 + \frac{1}{\epsilon} \sum f_i\Big).
\end{equation}
\end{corollary}

\begin{proof}
By the concavity of $f$,
\[ f ({\bf 1}) -  f (\lambda) \leq \sum f_i (1 - \lambda_i). \]
Therefore, if $\lambda_r \geq 0$ then
\[ f_r \lambda_r \leq f (\lambda) - f ({\bf 1}) +  \sum f_i
     + \sum_{\lambda_i < 0} f_i  |\lambda_i|
     \leq \epsilon \sum_{\lambda_i < 0} f_i  \lambda_i^2
          + \frac{C}{\epsilon} \sum f_i + C. \]

Suppose $\lambda_r < 0$. By Lemma~\ref{lem1.10} we have
\[ \sum f_i |\lambda_i|
   \leq \frac{\epsilon}{n} \sum f_i \lambda_i^2 + \frac{n}{4 \epsilon} \sum f_i
   \leq \epsilon \sum_{i\neq r} f_i \lambda_i^2 + \frac{C}{\epsilon} \sum f_i. \]
This proves \eqref{C210}.
\end{proof}

\bigskip

\section{Global bounds for the second derivatives}
\label{hess-g}
\setcounter{equation}{0}

\medskip

The goal of this section is to prove \eqref{hess-a10} under the hypotheses
(\ref{3I-20}), (\ref{3I-30}), (\ref{3I-40'}) and \eqref{3I-200}.
We start with a brief explanation of our notation and
basic formulas needed.
Throughout the paper $\nabla$ denotes the Levi-Civita connection
of  $(M^n, g)$. The curvature tensor is defined by
\[ R (X, Y) Z = - \nabla_X \nabla_Y Z + \nabla_Y \nabla_X Z
                 + \nabla_{[X, Y]} Z. \]

Let $e_1, \ldots, e_n$ be local frames on $M^n$ and denote
$g_{ij} = g (e_i, e_j)$, $\{g^{ij}\} = \{g_{ij}\}^{-1}$,
and $\nabla_i = \nabla_{e_i}$,
$\nabla_{ij} = \nabla_i \nabla_j - \nabla_{\nabla_i e_j}$, etc.
Define $ R_{ijkl}$, $R^i_{jkl}$ and 
 $\Gamma_{ij}^k$ respectively by
\[ 
   R_{ijkl} = \langle R (e_k, e_l) e_j, e_i \rangle, \;\;
   R^i_{jkl} = g^{im} R_{mjkl}, \;\; \nabla_i e_j = \Gamma_{ij}^k e_k.  \]

For a differentiable function $v$ defined on $M^n$, we
identify $\nabla v$ with the gradient of $v$, and
$\nabla^2 v$ denotes the Hessian of $v$ which is given by
$\nabla_{ij} v = \nabla_i (\nabla_j v) - \Gamma_{ij}^k \nabla_k v$.
Recall that $\nabla_{ij} v =\nabla_{ji} v$ and
\begin{equation}
\label{hess-A70}
 \nabla_{ijk} v - \nabla_{jik} v = R^l_{kij} \nabla_l v,
\end{equation}
\begin{equation}
\label{hess-A75}
\nabla_{ijkl} v - \nabla_{ikjl} v =
R^m_{ljk} \nabla_{im} v + \nabla_i R^m_{ljk} \nabla_m v,
\end{equation}
\begin{equation}
\label{hess-A76}
\nabla_{ijkl} v - \nabla_{jikl} v =
R^m_{kij} \nabla_{ml} v + R^m_{lij} \nabla_{km} v.
\end{equation}
From (\ref{hess-A75}) and (\ref{hess-A76}) we obtain
\begin{equation}
\label{hess-A80}
\begin{aligned}
\nabla_{ijkl} v - \nabla_{klij} v
= R^m_{ljk} \nabla_{im} v & + \nabla_i R^m_{ljk} \nabla_m v
      + R^m_{lik} \nabla_{jm} v \\
  & + R^m_{jik} \nabla_{lm} v
      + R^m_{jil} \nabla_{km} v + \nabla_k R^m_{jil} \nabla_m v.
\end{aligned}
\end{equation}

Let $u \in C^4 (M)$ be an admissible solution of equation~\eqref{3I-10}.
Under orthonormal local frames $e_1, \ldots, e_n$,
equation~\eqref{3I-10} is expressed in the form
\begin{equation}
\label{3I-10'}
F (U_{ij}) := f (\lambda [U_{ij}]) = \psi
\end{equation}
where $U_{ij} = \nabla_{ij} u +  \chi_{ij}$. For simplicity, we shall still write equation~\eqref{3I-10} in the form \eqref{3I-10'} even if $e_1, \ldots, e_n$
are not necessarily orthonormal, although more precisely it should be
\[ F (\gamma^{ik} U_{kl} \gamma^{lj}) = \psi \]
where $\{\gamma^{ij}\}$ is the square root of $\{g^{ij}\}$:
$\gamma^{ik} \gamma^{kj} = g^{ij}$; as long as we use covariant derivatives
whenever we differentiate the equation it will make no difference.

We now begin the proof of \eqref{hess-a10}.
Let
\[ W = \max_{x \in \bar{M}} \max_{\xi \in T_x M^n, |\xi| = 1}
  (\nabla_{\xi \xi} u + \chi (\xi, \xi)) e^{\eta} \]
where $\eta$ is a function to be determined.
Suppose $W > 0$ and is achieved at an interior point
$x_0 \in M$ for some unit vector $\xi \in T_{x_0}M^n$.
Choose smooth orthonormal local frames $e_1, \ldots, e_n$
about $x_0$ such that $e_1(x_0) = \xi$
and $\{U_{ij} (x_0)\}$ is diagonal. We may also assume 
that $\nabla_i e_j = 0$ and therefore $\Gamma_{ij}^k = 0$ at $x_0$
for all $1 \leq i,j,k \leq n$.
At the point $x_0$ where the function
$\log U_{11} + \eta$
(defined near $x_0$) attains its maximum,
we have for $i = 1, \ldots, n$,
\begin{equation}
\label{hess-a30}
\frac{\nabla_i U_{11}}{U_{11}} + \nabla_i \eta = 0,
\end{equation}
\begin{equation}
\label{hess-a40}
\begin{aligned}
\frac{\nabla_{ii} U_{11}}{U_{11}}
   - \Big(\frac{\nabla_i U_{11}}{U_{11}}\Big)^2 + \nabla_{ii} \eta \leq 0.
\end{aligned}
\end{equation}

Here we wish to add some explanations which might be helpful to the reader. 
First we note that $U_{1j} (x_0) = 0$ for $j \geq 2$ so $\{U_{ij} (x_0)\}$
can be diagonalized. 
To see this let $e^{\theta} = e_1 \cos \theta + e_j \sin \theta$. Then
\[ U_{e^{\theta} e^{\theta}} (x_0) = U_{11} \cos^2 \theta + 
     2 U_{1j} \sin \theta \cos \theta +  U_{jj} \sin^2 \theta \]
has a maximum at $\theta = 0$. Therefore,
\[ \frac{d}{d \theta} U_{e^{\theta} e^{\theta}} (x_0)\Big|_{\theta = 0} = 0. \]
This gives $U_{1j} (x_0) = 0$.

Next, at $x_0$ we have
\begin{equation}
\label{a30a}
   \nabla_i (U_{11})  = \nabla_i U_{11}, 
 \end{equation} 
that is 
$ e_i (U_{11}) = \nabla_i U_{11} 
  \equiv \nabla^3 u (e_1, e_1, e_i) + \nabla \chi (e_1, e_1, e_i)$,
and   
\begin{equation}
\label{a40a}     
\nabla_{ij} (U_{11}) = \nabla_{ij} U_{11}.
\end{equation}

One can see \eqref{a30a} immediately if we assume $\Gamma_{ij}^k = 0$ at $x_0$
for all $1 \leq i, j, k \leq n$. In general, we have
\[ \nabla_i (U_{11}) 
                     = \nabla_i U_{11} + 2 \Gamma_{i1}^k U_{1k}
                     = \nabla_i U_{11} + 2 \Gamma_{i1}^1 U_{11} \]
as $U_{1k} (x_0) = 0$. On the other hand, since $e_1, \ldots e_n$ are orthonormal, 
\[ g (\nabla_k e_i,  e_j) + g (e_i, \nabla_k e_j) = 0 \]
and
\[  g (\nabla_i e_1, \nabla_j e_1) + g (e_1, \nabla_i \nabla_j e_1) = 0. \]
Thus  
\begin{equation}
\label{a50a} 
\Gamma_{ki}^j + \Gamma_{kj}^i = 0
\end{equation}
 and
\[ \Gamma_{i1}^k \Gamma_{j1}^k + \nabla_i (\Gamma_{j1}^1) 
    + \Gamma_{j1}^k \Gamma_{ik}^1 = 0. \]
    This gives $\Gamma_{i1}^1 = 0$ and $\nabla_i (\Gamma_{j1}^1) = 0$. 
So we have \eqref{a30a}.

For \eqref{a40a} we calculate directly, 
\[ \begin{aligned}
 \nabla_{ij} (U_{11}) 
   = \,& \nabla_i (\nabla_{j} (U_{11})) - \Gamma_{ij}^k \nabla_k (U_{11}) \\
   = \,& \nabla_i (\nabla_{j} U_{11} + 2 \Gamma_{j1}^k U_{1k}) 
          - \Gamma_{ij}^k \nabla_k U_{11} \\
   = \,& \nabla_{ij} U_{11} + \Gamma_{ij}^k \nabla_k U_{11} 
          + 2 \Gamma_{i1}^k \nabla_{j} U_{1k} + 2 \nabla_i (\Gamma_{j1}^k) U_{1k} \\
     \,& + 2 \Gamma_{j1}^k \nabla_i U_{1k} + 2 \Gamma_{j1}^k \Gamma_{i1}^l U_{lk} 
          + 2 \Gamma_{j1}^k \Gamma_{ik}^l U_{1l}- \Gamma_{ij}^k \nabla_k U_{11} \\
   = \,& \nabla_{ij} U_{11} + 2 \Gamma_{i1}^k \nabla_{j} U_{1k} 
          + 2 \Gamma_{j1}^k \nabla_i U_{1k} + 2 \Gamma_{i1}^k \Gamma_{j1}^k U_{kk}
          - 2  \Gamma_{i1}^k \Gamma_{j1}^k U_{11}
\end{aligned} \] 
by \eqref{a50a} and $\nabla_i (\Gamma_{j1}^1) = 0$. Therefore we have \eqref{a40a} if $\Gamma_{ij}^k = 0$ at $x_0$.

We now continue our proof of \eqref{hess-a10}.
Differentiating equation (\ref{3I-10'}) twice, we obtain at $x_0$,
\begin{equation}
\label{hess-a60}
F^{ij} \nabla_k U_{ij} =  \nabla_k \psi,
\;\;\; \mbox{for all $k$},
\end{equation}
\begin{equation}
\label{hess-a65}
\begin{aligned}
F^{ii} \nabla_{11} U_{ii} + \sum F^{ij, kl} \nabla_1 U_{ij} \nabla_1 U_{kl}
 = \nabla_{11} \psi.
\end{aligned}
\end{equation}
Here and throughout rest of the paper,  $F^{ij} = F^{ij} (\{U_{ij}\})$.
By (\ref{hess-A80}), 
\begin{equation}
\label{hess-a68}
\begin{aligned}
F^{ii} \nabla_{ii} U_{11}
  \geq \,& F^{ii} \nabla_{11} U_{ii}
          + 2 F^{ii} R_{1i1i} (\nabla_{11} u - \nabla_{ii} u)
          - C \sum F^{ii} \\
 \geq \,& F^{ii} \nabla_{11} U_{ii} - C (1 + U_{11}) \sum F^{ii}.
\end{aligned}
\end{equation}
Here we note that $C$ depends on the gradient bound $|\nabla u|_{C^0 (\bM)}$.
 From (\ref{hess-a40}), (\ref{hess-a65}) and (\ref{hess-a68})
we derive
\begin{equation}
\label{hess-a71}
\begin{aligned}
U_{11} F^{ii} \nabla_{ii} \eta
  \leq \,& E - \nabla_{11} \psi +  C (1 + U_{11}) \sum F^{ii}
\end{aligned}
\end{equation}
where
\[ E \equiv  F^{ij, kl} \nabla_1 U_{ij} \nabla_1 U_{kl}
           + \frac{1}{U_{11}} F^{ii} (\nabla_i U_{11})^2. \]

To estimate $E$ we follow the idea of Urbas~\cite{Urbas02}.
Let $0 < s < 1$  (to be chosen) and
\[  \begin{aligned}
J \,& = \{i: U_{ii} \leq - s U_{11}\}, \;\;
K  = \{i> 1: U_{ii} > - s U_{11}\}.
  \end{aligned} \]
It was shown by Andrews~\cite{Andrews94}
and Gerhardt~\cite{Gerhardt96} (see also \cite{Urbas02}) that
\[ - F^{ij, kl} \nabla_1 U_{ij} \nabla_1 U_{kl}
 \geq \sum_{i \neq j} \frac{F^{ii} - F^{jj}}{U_{jj} - U_{ii}}
            (\nabla_1 U_{ij})^2. \]
Therefore,
\begin{equation}
\label{gsz-G270}
\begin{aligned}
 - F^{ij, kl} \nabla_1 U_{ij} \nabla_1 U_{kl}
 \geq \,& 2 \sum_{i \geq 2} \frac{F^{ii} - F^{11}}{U_{11} - U_{ii}}
            (\nabla_1 U_{i1})^2 \\
 \geq \,& 2 \sum_{i \in K} \frac{F^{ii} - F^{11}}{U_{11} - U_{ii}}
            (\nabla_1 U_{i1})^2 \\
 \geq \,& \frac{2}{(1+s) U_{11}} \sum_{i \in K} (F^{ii} - F^{11})
            (\nabla_1 U_{i1})^2 \\
 \geq \,& \frac{2 (1-s)}{(1+s) U_{11}} \sum_{i \in K} (F^{ii} - F^{11})
           [(\nabla_i U_{11})^2 - C/s].
\end{aligned}
\end{equation}
We now fix $s \leq 1/3$ and hence
\[ \frac{2 (1 - s)}{1 + s} \geq 1. \]
From \eqref{gsz-G270} and \eqref{hess-a30} it follows that
\begin{equation}
\label{hess-a271}
\begin{aligned}
 E \leq \,& \frac{1}{U_{11}} \sum_{i \in J} F^{ii} (\nabla_i U_{11})^2
             + \frac{C}{U_{11}} \sum_{i \in K} F^{ii}
             + \frac{C F^{11}}{U_{11}}  \sum_{i \notin J} (\nabla_i U_{11})^2  \\
   \leq \,& U_{11} \sum_{i \in J} F^{ii} (\nabla_i \eta)^2
             + \frac{C}{U_{11}} \sum F^{ii}
             + C U_{11} F^{11} \sum_{i \notin J} (\nabla_i \eta)^2.
\end{aligned}
\end{equation}

Let
\[ \eta = \phi (|\nabla u|^2) + a (\ul{u} - u) \]
where $\phi$ is a positive function, $\phi' > 0$, 
and $a$ is a positive constant. We calculate
\[ \begin{aligned}
 \nabla_i \eta
   = \,& 2 \phi' \nabla_{k} u \nabla_{ik} u + a \nabla_i (\ul{u} - u) \\
   = \,& 2 \phi' (U_{ii} \nabla_i u - \chi_{ik} \nabla_{k} u)
        + a \nabla_i (\ul{u} - u),
 \end{aligned}   \]
\[  \begin{aligned}
  \nabla_{ii} \eta
   = \,&  2 \phi' (\nabla_{ik} u \nabla_{ik} u
          + \nabla_{k} u \nabla_{iik} u)
          + 2 \phi'' (\nabla_{k} u \nabla_{ik} u)^2 
          + a \nabla_{ii} (\ul{u} - u).
\end{aligned} \]
Therefore,
\begin{equation}
\label{hess-a272}
\begin{aligned}
  \sum_{i \in J} F^{ii} (\nabla_i \eta)^2
    \leq \,& 8 (\phi')^2 \sum_{i \in J} F^{ii} (\nabla_{k} u \nabla_{ik} u)^2
          + C a^2 \sum_{i \in J} F^{ii},
\end{aligned}
\end{equation}
\begin{equation}
\label{hess-a272.5}
 \sum_{i  \notin J} (\nabla_i \eta)^2
\leq C (\phi')^2 U_{11}^2 + C (\phi')^2 + C a^2
\end{equation}
and by \eqref{hess-a60},
 \begin{equation}
\label{hess-a273}
 \begin{aligned}
 F^{ii} \nabla_{ii} \eta
\geq \,&  \phi'  F^{ii} U_{ii}^2 + 2 \phi'' F^{ii} (\nabla_{k} u \nabla_{ik} u)^2 \\
       &  + a F^{ii} \nabla_{ii} (\ul{u} - u) - C \phi' \Big(1 + \sum F^{ii}\Big).
\end{aligned}
\end{equation}

Let
$\phi (t) = b (1+t)^2$;
we may assume $\phi'' - 4 (\phi')^2 = 2 b (1 - 8 \phi) \geq 0$ in
any fixed interval $[0, C_1]$ by requiring $b > 0$ sufficiently small.
Combining \eqref{hess-a71}, \eqref{hess-a271}, \eqref{hess-a272},
\eqref{hess-a272.5} and \eqref{hess-a273},
we obtain
 \begin{equation}
\label{hess-a276}
 \begin{aligned}
\phi' F^{ii} U_{ii}^2 +  a F^{ii} \nabla_{ii} (\ul{u} - u)
    \leq \,& C a^2 \sum_{i \in J} F^{ii} + C ((\phi')^2 U_{11}^2 + A^2) F^{11} \\
            & - \frac{\nabla_{11} \psi}{U_{11}} + C \Big(1 +  \sum F^{ii}\Big).
\end{aligned}
\end{equation}

Suppose $U_{11} (x_0) > R$ sufficiently large and apply Theorem~\ref{3I-th3'} to
$A = \{\nabla_{ij} \ul{u} + \chi_{ij}\}$ and $B= \{U_{ij}\}$ at $x_0$.
We see that
\[ F^{ii} \nabla_{ii} (\ul{u} - u)
  = F^{ii} [(\nabla_{ii} \ul{u} + \chi_{ii}) - U_{ii}]
   \geq \theta \Big(1 + \sum F^{ii}\Big). \]
Plug this into \eqref{hess-a276} and fix $a$ sufficiently large;
since $|\nabla_{11} \psi| \leq C U_{11}$ if $\psi = \psi (x, u)$
we derive
\begin{equation}
\label{hess-a276'}
  \phi' F^{ii} U_{ii}^2
    \leq C a^2 \sum_{i \in J} F^{ii} + C ((\phi')^2 U_{11}^2 + a^2) F^{11}.
\end{equation}
Note that
\begin{equation}
\label{hess-a274}
 F^{ii} U_{ii}^2 \geq  F^{11} U_{11}^2 + \sum_{i \in J} F^{ii} U_{ii}^2
   \geq F^{11} U_{11}^2 + s^2 U_{11}^2 \sum_{i \in J} F^{ii}.
 \end{equation}
Fixing $b$ sufficiently small
we obtain from \eqref{hess-a276'} a bound $U_{11} \leq C a/\sqrt{b}$.
This implies \eqref{hess-a10}, and \eqref{hess-a10c} when $M$ is closed.

\bigskip

\section{Boundary estimates}
\label{hess-b}
\setcounter{equation}{0}

\medskip

In this section we establish the boundary estimate \eqref{hess-a10b}
under the assumptions of Theorem~\ref{3I-th4b}. Throughout this section
we assume the function $\varphi \in C^4 (\partial M)$ is extended to
a $C^4$ function on $\bM$, still denoted $\varphi$.

For a point $x_0$ on $\partial M$, we shall choose
smooth orthonormal local frames $e_1, \ldots, e_n$ around $x_0$ such that
when restricted to $\partial M$, $e_n$ is normal to $\partial M$.

Let $\rho (x)$ denote the distance
from $x$ to $x_0$,
\[ \rho (x) \equiv \mbox{dist}_{M^n} (x, x_0), \]
and $M_{\delta} = \{x \in M : \rho (x) < \delta \}$.
Since $\partial M$ is smooth we may assume the distance function
to $\partial M$
\[ d (x) \equiv \mbox{dist} (x, \partial M) \]
is smooth in $M_{\delta_0}$ for fixed $\delta_0 > 0$ sufficiently small
(depending only on the curvature of $M$ and the principal curvatures of
$\partial M$.)
Since $\nabla_{ij} \rho^2 (x_0) = 2 \delta_{ij}$,
we may assume $\rho$ is smooth in $M_{\delta_0}$ and
\begin{equation}
  \{\delta_{ij}\} \leq \{\nabla_{ij} \rho^2\} \leq 3 \{\delta_{ij}\}
\;\;\; \mbox{in} \;\;\; M_{\delta_0}.
\label{cns1-E125}
\end{equation}

The following lemma which crucially depends on Theorem~\ref{3I-th3'}
 plays key roles in our boundary estimates.

\begin{lemma}
\label{ma-lemma-E10}
There exist some uniform positive constants
$t, \delta, \varepsilon$ sufficiently small and $N$ sufficiently large
such that the function
\begin{equation}
\label{ma-E85}
v = (u - \ul{u}) + t d - \frac{N d^2}{2}
\end{equation}
satisfies $v \geq 0$ on $\bar{M_{\delta}}$ and
\begin{equation}
\label{ma-E86}
F^{ij} \nabla_{ij} v \leq - \varepsilon \Big(1 + \sum F^{ii}\Big)
   \;\; \mbox{in $M_{\delta}$}.
\end{equation}
\end{lemma}

\begin{proof}
We note that to ensure $v \geq 0$ in $\bar{M_{\delta}}$
we may require $\delta \leq 2t/N$ after $t, N$ being fixed.
Obviously,
\begin{equation}
\label{eq-110}
\begin{aligned}
F^{ij} \nabla_{ij} v
   = \,& F^{ij} \nabla_{ij} (u - \ul{u}) + (t - N d)  F^{ij} \nabla_{ij} d
     - N  F^{ij} \nabla_{i} d \nabla_{j} d \\
\leq \,& C_1 (t+N d) \sum F^{ii} + F^{ij}  \nabla_{ij} (u - \ul{u})
        - N  F^{ij} \nabla_{i} d \nabla_{j} d.
\end{aligned}
\end{equation}

Fix $\varepsilon > 0$ sufficiently small and $R \geq R_A$ so that
Theorem~\ref{3I-th3'} holds for $A = \{\nabla_{ij} \ul{u} + \chi_{ij}\}$
and $B = \{U_{ij}\}$ at every point in $\bar{M}_{\delta_0}$.
Let $\lambda = \lambda [\{U_{ij}\}]$ be the eigenvalues of $\{U_{ij}\}$.
 At a fixed point in $M_{\delta}$ we consider two cases:
(a) $|\lambda| \leq R$; and (b) $|\lambda| > R$.

In case (a) there are uniform bounds (depending on $R$)
\[ 0 < c_1 \leq \{F^{ij}\} \leq C_1 \]
and therefore $F^{ij} \nabla_{i} d \nabla_{j} d \geq c_1$
since $|\nabla d| \equiv 1$. We may fix $N$ large enough
so that \eqref{ma-E86} holds for any $t, \varepsilon \in (0,1]$,
as long as $\delta$ is sufficiently small.

In case (b) by Theorem~\ref{3I-th3'} and \eqref{eq-110} we may further
require $t$ and $\delta$ so that \eqref{ma-E86} holds for some
different (smaller) $\varepsilon > 0$.
\end{proof}

We now start the proof of \eqref{hess-a10b}.
Consider a point $x_0 \in \partial M$.
Since $u - \ul{u} = 0$ on $\partial M$ we have
\begin{equation}
\label{hess-a200}
\nabla_{\alpha \beta} (u - \ul{u})
 = -  \nabla_n (u - \ul{u}) \varPi (e_{\alpha}, e_{\beta}), \;\;
\forall \; 1 \leq \alpha, \beta < n \;\;
\mbox{on  $\partial M$}
\end{equation}
where 
$\varPi$ denotes the second fundamental form of $\partial M$.
Therefore,
\begin{equation}
\label{hess-E130}
|\nabla_{\alpha \beta} u| \leq  C,  \;\; \forall \; 1 \leq \alpha, \beta < n
\;\;\mbox{on} \;\; \partial M.
\end{equation}

To estimate the mixed tangential-normal and pure normal second
derivatives we note the following formula
\[ \nabla_{ij} (\nabla_{k} u)
  = \nabla_{ijk} u + \Gamma_{ik}^l \nabla_{jl} u + \Gamma_{jk}^l \nabla_{il} u
  +  \nabla_{\nabla_{ij} e_k} u. \]
By \eqref{hess-a60}, therefore,
\begin{equation}
\label{hess-E170}
\begin{aligned}
|F^{ij} \nabla_{ij} \nabla_k (u - \varphi)|
\leq \,& 2 F^{ij}  \Gamma_{ik}^l \nabla_{jl} u  + C \Big(1 + \sum F^{ii}\Big) \\
\leq \,& C \Big(1 + \sum f_i |\lambda_i| + \sum f_i\Big).
\end{aligned}
\end{equation}

Let
\begin{equation}
\label{hess-E176}
 \varPsi
   = A_1 v + A_2 \rho^2 - A_3 \sum_{\beta < n} |\nabla_{\beta} (u - \varphi)|^2.
\end{equation}
By \eqref{hess-E170} we have
\begin{equation}
\begin{aligned}
 F^{ij} \nabla_{ij} |\nabla_{\beta} (u - \varphi)|^2
   = \,& 2  F^{ij} \nabla_{\beta} (u - \varphi)
           \nabla_{ij} \nabla_{\beta} (u - \varphi) \\
     \,& + 2  F^{ij} \nabla_i \nabla_{\beta} (u - \varphi)
               \nabla_j \nabla_{\beta} (u - \varphi)  \\
\geq \,& F^{ij} U_{i \beta} U_{j \beta}
         - C \Big(1 + \sum f_i |\lambda_i| + \sum f_i\Big).
\end{aligned}
\end{equation}

For fixed $ 1 \leq \alpha < n$, by Lemma~\ref{ma-lemma-E10},
Proposition~\ref{prop5.2}
and Corollary~\ref{cor1.10} we see that
\begin{equation}
\label{hess-E170'}
F^{ij} \nabla_{ij} (\varPsi \pm \nabla_{\alpha} (u - \varphi)) \leq 0, \;\;
\forall \; \;\; \mbox{in $M_{\delta}$}
\end{equation}
and $\varPsi \pm \nabla_{\alpha} (u - \varphi) \geq 0$ on $\partial M_{\delta}$
when $A \gg A_2 \gg A_3 \gg 1$. By the maximum principle we derive
$\varPsi \pm \nabla_{\alpha} (u - \varphi) \geq 0$
in $M_{\delta}$ and therefore
\begin{equation}
\label{hess-E130'}
|\nabla_{n\alpha} u (x_0)| \leq \nabla_n \varPsi (x_0) \leq C,
\;\; \forall \; \alpha < n.
\end{equation}

It remains to derive
\begin{equation}
\label{cma-200}
 \nabla_{n n} u (x_0) \leq C.
\end{equation}
Following an idea of Trudinger~\cite{Trudinger95}
we show that there are uniform constants $c_0, R_0$
such that for all $R > R_0$, $(\lambda' [\{U_{\alpha \beta} (x_0)\}], R) \in \Gamma$ and
\[ f (\lambda' [\{U_{\alpha \beta} (x_0)\}], R) \geq \psi (x_0) + c_0 \]
where
$\lambda' [\{U_{\alpha \beta}\}] = (\lambda'_1, \cdots, \lambda'_{n-1})$
denotes the eigenvalues of the $(n-1) \times (n-1)$ matrix
$\{U_{\alpha \beta}\}$ ($1 \leq \alpha, \beta \leq n-1$).
Suppose we have found such $c_0$ and $R_0$.
By Lemma 1.2 of \cite{CNS3}, from estimates \eqref{hess-E130} and
\eqref{hess-E130'} we can find $R_1 \geq R_0$ such
that if $U_{nn} (x_0) > R_1$,
\[ f (\lambda [\{U_{ij} (x_0)\}])
    \geq f (\lambda' [\{U_{\alpha \beta} (x_0)\}], U_{nn}(x_0)) - \frac{c_0}{2}. \]
By equation~\eqref{3I-10} this gives a desired bound $U_{nn} (x_0) \leq R_1$
for otherwise, we would have
\[ f (\lambda [\{U_{ij} (x_0)\}]) \geq \psi (x_0) + \frac{c_0}{2}. \]

For $R > 0$ and a symmetric $(n-1)^2$
matrix $\{r_{\alpha  {\beta}}\}$ with $(\lambda' [\{r_{\alpha \beta} (x_0)\}], R) \in \Gamma$ , define
\[ \tF [r_{\alpha \beta}]
   \equiv f (\lambda' [\{r_{\alpha \beta}\}], R) \]
and consider
\[ m_R \equiv \min_{x_0 \in \partial M} 
         \tF [U_{\alpha \beta} (x_0)] - \psi (x_0). \]
Note that  $\tF$ is concave and $m_R$ is increasing in $R$ by \eqref{3I-20},
and that
\[  c_R \equiv \inf_{\partial M}
  (\tF[\ul{U}_{\alpha {\beta}}] - F [\ul{U}_{ij}])
      \geq \inf_{\partial M} (\tF[\ul{U}_{\alpha {\beta}}] -
            f (\lambda' [\ul{U}_{\alpha {\beta}}], \ul{U}_{nn})) > 0 \]
when $R$ is sufficiently large.

We wish to show $m_R > 0$ for $R$ sufficiently large.
Suppose $m_R$ is achieved at a point $x_0 \in \partial M$.
Choose local orthonormal frames around $x_0$ as before and let
\[ \tF^{\alpha {\beta}}_0
 = \frac{\partial \tF}{\partial r_{\alpha {\beta}}}
        [U_{\alpha {\beta}} (x_0)]. \]
Since $\tF$ is concave,
for any symmetric matrix $\{r_{\alpha  \beta}\}$ with
$(\lambda' [\{r_{\alpha  \beta}\}], R) \in \Gamma$,
\begin{equation}
\label{c-200}
\tF^{\alpha {\beta}}_0 (r_{\alpha  {\beta}} - U_{\alpha {\beta}} (x_0))
    \geq \tF [r_{\alpha  {\beta}}] - \tF [U_{\alpha  {\beta}} (x_0)].
 \end{equation}
In particular,
\begin{equation}
\label{c-210}
\tF^{\alpha {\beta}}_0 U_{\alpha  {\beta}} - \psi
- \tF^{\alpha {\beta}}_0 U_{\alpha  {\beta}} (x_0) + \psi (x_0)
\geq \tF [U_{\alpha  {\beta}}] - \psi - m_0 \geq 0 \;\;
\mbox{on $\partial M$}.
\end{equation}

By \eqref{hess-a200} we have on $\partial M$,
\begin{equation}
\label{c-220}
 U_{\alpha {\beta}} = \ul{U}_{\alpha {\beta}}
    - \nabla_n (u - \ul{u}) \sigma_{\alpha {\beta}}
\end{equation}
where
$\sigma_{\alpha {\beta}} = \langle \nabla_{\alpha} e_{\beta}, e_n \rangle$;
note that
$\sigma_{\alpha \beta} = \varPi (e_\alpha, e_\beta)$ on
$\partial M$. 
It follows that
\[ \begin{aligned}
 \nabla_n (u - \ul{u}) \tF^{\alpha {\beta}}_0 \sigma_{\alpha {\beta}} (x_0)
   = \,& \tF^{\alpha {\beta}}_0 ( \ul{U}_{\alpha {\beta}} (x_0) -
         U_{\alpha  {\beta}} (x_0)) \\
\geq \,& \tF[\ul{U}_{\alpha {\beta}} (x_0)] - \tF[U_{\alpha {\beta}} (x_0)] \\
  =  \,& \tF[\ul{U}_{\alpha {\beta}} (x_0)] - \psi(x_0) - m_R \geq c_R - m_R.
\end{aligned} \]
Consequently, if
\[ \nabla_n (u - \ul{u}) (x_0)
    \tF^{\alpha {\beta}}_0 \sigma_{\alpha {\beta}} (x_0)
   \leq c_R/2 \]
then $m_R \geq c_R/2$ and we are done.

Suppose now that
\[ \nabla_n (u - \ul{u}) (x_0) \tF^{\alpha {\beta}}_0 \sigma_{\alpha {\beta}} (x_0)
    > \frac{c_R}{2} \]
and let $\eta \equiv \tF^{\alpha {\beta}}_0 \sigma_{\alpha {\beta}}$. Note that
\begin{equation}
\label{c-230}
\eta (x_0) \geq c_R/2 \nabla_n (u - \ul{u}) (x_0) \geq 2 \epsilon_1 c_R
\end{equation}
for some uniform $\epsilon_1 > 0$ independent of $R$.
We may assume $\eta \geq \epsilon_1 c_R$ on $\bar{M_{\delta}}$
by requiring $\delta$ small.
Define in $M_{\delta}$,
\[ \begin{aligned}
\varPhi
     = \,& - \nabla_n (u - \varphi)
        + \frac{1}{\eta} \tF^{\alpha {\beta}}_0
           (\nabla_{\alpha {\beta}} \varphi + \chi_{\alpha {\beta}}
          - U_{\alpha {\beta}} (x_0)) - \frac{\psi - \psi (x_0)}{\eta}  \\
\equiv \,& - \nabla_n (u - \varphi) + Q.
\end{aligned} \]
We have $\varPhi (0) = 0$ and
$\varPhi \geq 0$ on $\partial M$ near $0$ by \eqref{c-210} since
\[  \nabla_{\alpha {\beta}} u  = \nabla_{\alpha {\beta}} \varphi -
    \nabla_n (u - \varphi) \sigma_{\alpha {\beta}} \;\;
\mbox{on $\partial M$}, \]
while by \eqref{hess-E170},
\begin{equation}
\label{gblq-B360}
 \begin{aligned}
F^{ij} \nabla_{ij} \varPhi
  \leq \, & - F^{ij} \nabla_{ij} \nabla_n u + C \sum F^{ii}
  \leq   C \Big(1 + \sum f_i |\lambda_i| + \sum f_i\Big).
\end{aligned}
\end{equation}

Consider the function $\varPsi$ defined in \eqref{hess-E176}.
Applying Lemma~\ref{ma-lemma-E10}, Proposition~\ref{prop5.2} and
Corollary~\ref{cor1.10} as before
 for $A_1 \gg A_2 \gg A_3 \gg 1$ we derive
$\varPsi + \varPhi \geq 0$ on $\partial M_{\delta}$ and
\begin{equation}
\label{cma-106}
F^{ij} \nabla_{ij} (\varPsi + \varPhi) \leq  0 \;\; \mbox{in $M_{\delta}$}.
\end{equation}
By the maximum principle,
$\varPsi +  \varPhi \geq 0$ in $M_{\delta}$. Thus
$\varPhi_n (x_0) \geq - \nabla_n \varPsi (x_0) \geq -C$.
This gives
$ \nabla_{nn} u (x_0) \leq C$. 

So we have an {\em a priori}
upper bound for all eigenvalues of $\{U_{ij} (x_0)\}$.
Consequently, $\lambda [\{U_{ij} (x_0)\}]$ is contained in a
compact subset of $\Gamma$ by \eqref{3I-40}, and therefore
\[ m_R = \tF [U_{\alpha  {\beta}} (x_0)] - \psi (x_0) > 0 \]
when $R$ is sufficiently large.
This completes the proof of \eqref{hess-a10b}.

\bigskip

\section{Further results and remarks}
\label{3I-R}
\setcounter{equation}{0}

\medskip

\subsection{The results of Li \cite{LiYY90} and Urbas~\cite{Urbas02}}

In \cite{LiYY90} Li treated equation~\eqref{3I-10} with $\chi = g$ on closed
manifolds with nonnegative sectional curvature, and in various other situations.
His basic assumptions used in the second derivative estimates include
\eqref{3I-20}, \eqref{3I-30}, \eqref{3I-40'} as well as the following:
\begin{equation}
\label{3I-R10}
L_0 :=\lim_{\lambda \rightarrow 0, \lambda \in \Gamma} \inf f (\lambda) > - \infty,
\end{equation}
and
\begin{equation}
\label{3I-R20}
 \lim_{|\lambda| \rightarrow + \infty, \lambda \in \partial \Gamma^{\sigma}}
   \sum f_i (\lambda) = + \infty, \;\; \forall \, \sigma > \sup_{\partial \Gamma} f.
\end{equation}
Li also derived the gradient estimates under the same assumptions.

Urbas~\cite{Urbas02} was able remove the nonnegative curvature condition in \cite{LiYY90},
and showed that assumption~\eqref{3I-R20} could be replaced by
\begin{equation}
\label{3I-R20'}
   \sum f_i (\lambda) \geq \delta_{\sigma},
   \;\; \forall \, \lambda \in \partial \Gamma^{\sigma}, \;
   \sigma > \sup_{\partial \Gamma} f,
\end{equation}
and
\begin{equation}
\label{3I-R30}
 \lim_{|\lambda| \rightarrow + \infty, \lambda \in \partial \Gamma^{\sigma}}
   \sum f_i (\lambda) \lambda_i^2 = + \infty,
   \;\; \forall \, \sigma > \sup_{\partial \Gamma} f.
\end{equation}
The main assumption in ~\cite{Urbas02} for the gradient estimates is
\eqref{3I-R40}
which was also used in earlier papers for gradient estimates
\cite{Korevaar87}, \cite{LiYY91}, \cite{Trudinger90}, \cite{GS91}, \cite{CW01}.

The following lemma clarifies relations between assumptions \eqref{3I-R10},
\eqref{3I-R20} and \eqref{3I-200}.

\begin{lemma}
\label{3I-lemma-R10}
Suppose $f$ satisfies \eqref{3I-20}, \eqref{3I-30}, \eqref{3I-R10} and
\eqref{3I-R20}.
Then $\Gamma_n^+ \subset \mathcal{C}_{\sigma}^+$ for any
$\sigma > \sup_{\partial \Gamma} f$.
Consequently, condition~\eqref{3I-200} is satisfied if $\chi > 0$.
\end{lemma}

\begin{proof}
Let $\lambda \in \Gamma$. By the concavity of $f$,
\[ \sum f_{\lambda_i} (\lambda) (\delta - \lambda_i)
       \geq f (\delta {\bf 1}) - f (\lambda) \]
for any $\delta > 0$. Letting $\delta$ tend to $0$, we obtain by \eqref{3I-R10},
\begin{equation}
\label{3I-R50}
\sum f_{\lambda_i} (\lambda) \lambda_i \leq  f (\lambda) - L_0.
\end{equation}
Let $\mu 
\in \Gamma_n^+$ and assume
$\mu_1 \geq \cdots \geq \mu_n > 0$. Then for $\lambda \in \Gamma^{\sigma}$
\[ \sum f_{\lambda_i} (\lambda) (\mu_i - \lambda_i)
     \geq \mu_n \sum f_{\lambda_i} (\lambda) - \sum f_{\lambda} (\lambda) \lambda_i
     \geq \mu_n \sum f_{\lambda_i} (\lambda) + L_0 - \sigma > 0 \]
by \eqref{3I-R20} when $|\lambda|$ is sufficiently large. This clearly implies
$\mu \in \mathcal{C}_{\sigma}^+$.
\end{proof}

Concerning condition~\eqref {3I-R30} we have the following observation.

\begin{proposition}
\label{3I-propositio-R10}
Theorem~\ref{3I-th4b} still holds with
assumption~\eqref{3I-11s''} replaced by \eqref {3I-R30}, and therefore so does
Theorem~\ref{3I-thm2}.
\end{proposition}

\begin{proof}
In the function $\varPsi$ defined in \eqref{hess-E176} we replace $v$ by $(u - \ul{u})$
and call this new function $\tilde{\varPsi}$.
Since $\ul{u}$ is an admissible subsolution, by the concavity of $f$
there exists $\epsilon > 0$ such that
\[ F^{ij} \nabla_{ij} (\ul{u} - u) \geq \epsilon \sum F^{ii} - C. \]
Applying Proposition~\ref{prop5.2} and
Corollary~\ref{cor1.10}, by assumption~\eqref {3I-R30}  we may
choose  $A_1 \gg A_2 \gg A_3 \gg 1$ as before such that
\[ F^{ij} \nabla_{ij} \tilde{\varPsi} \leq - C \Big(1 + \sum f_i (1 + \lambda_i^2)\Big) \]
for any $C >0$ when $|\lambda|$ is sufficiently large.
The rest of the proof is now same as that of Theorem~\ref{3I-th4b}.
\end{proof}

\subsection{The gradient estimates}
Building upon the estimates in Theorems~\ref{3I-th4} and ~\ref{3I-th4b} with the aid of
Evans-Krylov theorem, one needs to derive {\em a prior} $C^1$ estimates
in order to establish existence of solutions to equation~\eqref{3I-10} either
on closed manifolds or for the Dirichlet problem on manifolds with boundary, using
standard analytic tools such as the continuity methods and degree arguments.
It seems an interesting question whether one can prove gradient
estimates under assumption~\eqref{3I-200}.
We wish to come back to the problem in future work. Here we only list some results that
were more or less already known to Li \cite{LiYY90} and Urbas~\cite{Urbas02}.

\begin{proposition}
\label{3I-propositio-R20}
Let $u \in C^3 (\bM)$ be an admissible solution of equation~\eqref{3I-10}
where $\psi \in C^1 (\bM)$. Suppose $f$ satisfies \eqref{3I-20}-\eqref{3I-40}.
Then
\begin{equation}
\label{3I-R60}
 \max_{\bM} |\nabla u| \leq C \big(1 + \max_{\partial M} |\nabla u|\big)
\end{equation}
where $C$ depends on $|u|_{C^0 (\bM)}$, under any of the following
additional assumptions:
{\bf (i)} $\Gamma = \Gamma_n^+$;
{\bf (ii$'$)} \eqref{3I-200}, $\psi_u \geq 0$ and that $(M, g)$ has nonnegative
sectional curvature;
{\bf (iii$'$)} \eqref{3I-55} and \eqref{3I-R40} for $|\lambda|$ sufficiently
large.
\end{proposition}

\begin{proof}
Consider case {\bf (i)}: $\Gamma = \Gamma_n^+$. For fixed $A > 0$ suppose
$A u + |\nabla u|^2$ has a maximum at an interior point $x_0 \in M$.
Then $A \nabla_i u + 2 \nabla_k u \nabla_{ki} u
= \nabla_k u (A \delta_{ki} + \nabla_{ki} u) = 0$ at $x_0$
for all $1 \leq i \leq n$.
This implies $\nabla u (x_0)= 0$  when $A$ is sufficiently large. Therefore,
\[ \sup_M |\nabla u|^2 \leq A \Big(\sup_{\partial M} u - \inf_M u\Big)
      + \sup_{\partial M} |\nabla u|^2. \]

Case {\bf (iii$'$)} was proved by Urbas~\cite{Urbas02} under the additional
assumption \eqref{3I-R20'} which is implied by \eqref{3I-55}. Indeed,
by the concavity of $f$ and \eqref{3I-55},
\[ A \sum f_{\lambda_i} (\lambda) \geq
     \sum f_{\lambda_i} (\lambda) \lambda_i +  f (A {\bf 1}) - f (\lambda)
     \geq  f (A {\bf 1}) - \sigma \]
for any $\lambda \in \Gamma$, $f (\lambda) = \sigma$.
Fixing $A$ sufficiently large
gives \eqref{3I-R20'}.

Case {\bf (ii$'$)}.
Gradient estimates were established by Li \cite{LiYY90} on closed
manifolds with nonnegative sectional curvature under the additional assumptions
\eqref{3I-R10} and \eqref{3I-R20}.
His proof can be modified to replace \eqref{3I-R10} and \eqref{3I-R20}
by \eqref{3I-200}. We only outline the proof.

Suppose  $|\nabla u|^2 e^{\phi}$
achieves a maximum at an interior point $x_0 \in M$.
Then at $x_0$,
\[ \frac{2 \nabla_k u \nabla_{ik} u}{|\nabla u|^2} + \nabla_i \phi = 0, \]
\[ 2 F^{ij} (\nabla_k u \nabla_{jik} + \nabla_{ik} u \nabla_{jk} u)
   + |\nabla u|^2 F^{ij} (\nabla_{ij} \phi - \nabla_i \phi \nabla_j \phi) \leq 0. \]
Following \cite{LiYY90} we use the nonnegative sectional curvature condition to
derive
\begin{equation}
\label{3I-R70}
 |\nabla u|^2 F^{ij} (\nabla_{ij} \phi - \nabla_i \phi \nabla_j \phi)
  \leq C |\nabla u| - \psi_u |\nabla u|^2.
\end{equation}
Now let $\phi = A (\ul{u} - u)^2 $ and fix $A > 0$ sufficiently small.
By \eqref{3I-200} and Theorem~\ref{3I-th3'} we derive a bound
$|\nabla u (x_0)| \leq C$
if $|\lambda [\nabla^2 u + \chi] (x_0)| \geq R$ for $R$ sufficiently large.

Suppose $|\lambda [\nabla^2 u + \chi] (x_0)| \leq R$. Then by \eqref{3I-20} and
\eqref{3I-40}, there exists $C_1 > 0$ depending on $R$ such that at $x_0$,
\[ \frac{g^{-1}}{C_1} \leq \{F^{ij}\} \leq C_1 g^{-1}. \]
From \eqref{3I-R70},
\[ \begin{aligned}
  \frac{C}{|\nabla u|}
  \,& \geq 2 A F^{ij} \nabla_{ij} (\ul{u} - u)
           + 2 A (1 - 2 A) F^{ij} \nabla_i (\ul{u} - u) \nabla_j (\ul{u} - u) \\
  \,& \geq 2 A (1 - 2 A) C_1^{-1} |\nabla (\ul{u} - u)|^2 - C A.
  \end{aligned}  \]
We derive a bound for $|\nabla u (x_0)|$ again.
\end{proof}

\subsection{An example}
\label{3I-R3}
Consider the function
\[ P_k (\lambda) := \prod_{i_1 < \cdots < i_k}
(\lambda_{i_1} + \cdots + \lambda_{i_k}), \;\; 1 \leq k \leq n\]
defined in the cone
\[ \mathcal{P}_k : = \{\lambda \in \bfR^n:
      \lambda_{i_1} + \cdots + \lambda_{i_k} > 0\}. \]
Obviously,
\[ \sup_{\partial \mathcal{P}_k} P_k = 0. \]

Let $f = \log P_k$. Then
\[ \frac{\partial f}{\partial \lambda_i} = \sum_{i_2 < \cdots < i_k; i_l \neq i}
    \frac{1}{\lambda_i + \lambda_{i_2} + \cdots + \lambda_{i_{k}}},  \]
\[ \frac{\partial^2 f}{\partial\lambda_i  \partial\lambda_j}
  = - \sum_{i_3 < \cdots < i_k; i_l \neq i, j}
   \frac{1}{(\lambda_i + \lambda_j + \lambda_{i_3} + \cdots + \lambda_{i_{k}})^2}. \]

Therefore $f = \log P_2$ satisfies \eqref{3I-20} and \eqref{3I-30}
in $\mathcal{P}_2$. Moreover, $\Gamma^{\sigma} \equiv \{P_2 > \sigma\}$ is strictly
convex and $\mathcal{C}_{\sigma}^+ = \mathcal{P}_2$.
Consequently, Corollary~\ref{3I-cor-R2} holds for $f =  P_2$.

In ~\cite{HS09} Huisken and Sinestrari studied the mean curvature flow of
hypersurfaces with principal curvatures
$(\kappa_1, \ldots, \kappa_n) \in \mathcal{P}_2$;
they call such hypersurfaces {\em two-convex}.

There seem interesting cases among the quotients $P_k/P_l$ but the situation is
more complicated. We hope to discuss them in future work.
Note that $P_1 = \sigma_n$, $P_n = \sigma_1$.

\end{document}